\documentclass{amsart}
\usepackage[a4paper]{geometry}
\usepackage{amsmath,amssymb,ascmac,amsthm,caption}
\usepackage{graphicx}
\usepackage{cases,bm}
\usepackage{tcolorbox}
\tcbuselibrary{breakable, skins, theorems}

\theoremstyle{plain}
\newtheorem{thm}{Theorem}[section]
 
 \newtheorem{lem}[thm]{Lemma}
 \newtheorem{prop}[thm]{Proposition}

\theoremstyle{definition}
 
 \newtheorem{fact}[thm]{Fact}
 \theoremstyle{remark}
 
 \newtheorem{ex}[thm]{Example}

\renewcommand{\proofname}{\textbf{Proof}}
\numberwithin{equation}{section}
\numberwithin{figure}{section}

\newcommand{\Ker}{\operatorname{Ker}}
\newcommand{\image}{\operatorname{Im}}
\newcommand{\sgn}{\operatorname{sgn}}

\begin{document}
\title[GEOMETRY OF CURVES ON WHITNEY UMBRELLA]
{Geometry of curves passing through Whitney umbrella}
\author{Hiroyuki Hayashi}
\address{Department of Mathematics, Graduate School of Science, Kobe University, 1-1, Rokkodai, Nada-ku, Kobe 657-8501, Japan}
\email{224s017s@stu.kobe-u.ac.jp}
\maketitle
%%%%%%%%%%%%%%%%%%%%%%%%%%%%%%%%%%%%%%

%%%%%%%%%%%%%start documents%%%%%%%%%%%%%%%%%
%% Abstract
\begin{abstract}
We study geometry {of} curves passing through the  Whitney umbrella {singularities} {by} using a Darboux frame. We define three functions using {a} Frenet-Serret type formula related to the geodesic curvature, the normal curvature, and the geodesic torsion. We investigate the degrees of divergence, the top-terms of these functions, and their geometric meanings. We also consider a developable surface along the curve.
\end{abstract}

%% main text
\section{Introduction}%%%%%%%%%%%%%%%%%%%%%%%%%%%%%%%%%%%%%%%%%%%%%%%%%%%%%%%%%%%%%%%%
Let $f:(\mathbb{R}^2,0)\to(\mathbb{R}^3,0)$ be a smooth map-germ that locally defines a surface, possibly with singularities, in $\mathbb{R}^3$. Two map-germs $f , g : (\mathbb{R}^{n_1} ,0) \to (\mathbb{R}^{n_2} , 0)$ are said to be $\mathcal{A}$-{\it equivalent} if there exist diffeomorphism-germs $\psi : (\mathbb{R}^{n_1} ,0) \to (\mathbb{R}^{n_1} ,0)$ and $\phi : (\mathbb{R}^{n_2} ,0) \to (\mathbb{R}^{n_2} ,0)$ such that $g = \phi \circ f \circ \psi ^{-1}$ holds. A map-germ $f: (\mathbb{R}^2 ,0) \to (\mathbb{R}^3 ,0)$ is called a {\it Whitney umbrella} (also known as a {\it cross-cap}) if $f$ is $\mathcal{A}$-equivalent to the map-germ:
\[
(u,v) \mapsto (u, uv, v^2)^T
\]
 at the origin, where $(\ )^T$ is the transpose.  The singular point $(0,0)$ of a Whitney umbrella is called a {\it Whitney umbrella singularity}. Whitney umbrella singularities most frequently appear on surfaces in $\mathbb{R}^3$. To study differential geometry on surfaces, unit normal vector fields play a central role. However, it is known that the unit normal vector field around a Whitney umbrella singularity cannot be smoothly extended. By applying a blowing-up, the unit normal vector field around the Whitney umbrella singularity may be extended beyond the singular point \cite{FH2012}. This fact strongly suggests that the normal vector field along a curve passing through a Whitney umbrella singular point can be smoothly extended beyond the singular point. One can take a Darboux frame along such a curve.

In this paper, we study geometry {of} curves passing through the Whitney umbrella singularity using a Darboux frame that consists of smoothly defined unit normal and tangent vectors along the curve. By using Frenet-Serret type formula, we define three functions. They are related to the geodesic curvature, the normal curvature, and the geodesic torsion, all of which are defined on a set of regular points. These curvatures generally diverge at singularities on surfaces. To investigate the properties of such curves and the geometric relationship between the Whitney umbrella and the curve, we examine the degrees and the top-terms of these functions. We give geometric meanings of the vanishings of the top-terms of the functions. Moreover, we consider a developable surface along the curve passing through the Whitney umbrella singularity, and we give degrees of divergence and top-terms of the invariants. Developable surfaces are classified into cylinders, cones, and tangent developable surfaces, as well as their gluings. We introduce pseudo-cylindrical developable surface and pseudo-conical developable surfaces. These classes have made them easier to handle developable surfaces that are neither cylinders nor cones.

\section{Preliminaries}%%%%%%%%%%%%%%%%準備%%%%%%%%%%%%%%%%%%%%%%%%%%%%%%%%%%%%%%%%%%%%%%%%%
{We recall the following fact} \cite{W1995}.
\begin{fact}\label{Fact}%%%%%%%%%%%%%%%%%%%%%%%%%%%%%%%%
Let $W : (\mathbb{R}^2 ,0) \to (\mathbb{R}^3 ,0)$ be a Whitney umbrella. Then for any $k\geq3$, there exist an orientation preserving diffeomorphism $\psi :(\mathbb{R}^2 ,0) \to (\mathbb{R}^2 ,0)$ and a rotation $T\in SO(3)$ such that \begin{align}\label{normalWU}
T \circ W \circ \psi (u,v)=
\left(
\begin{array}{c}
	u \\
	uv +B(u) + O(u,v)^{k+1}\\
	A(u,v) + O(u,v)^{k+1}
\end{array}
\right),
\end{align}
where 
\[
B(u)=\displaystyle\sum_{i=3}^{k}\frac{b_i}{i!}v^i,\ 
A(u,v) =\displaystyle\sum_{m=2}^k \sum_{i+j=m}\frac{a_{i j}}{i! j!}u^i v^j,\ 
b_i\in\mathbb{R} ,\ a_{i j}\in\mathbb{R}, \ a_{02} \neq 0,
\]
$O(u,v)^{k}$ is the terms whose degrees are greater than or equal to $k$.
\end{fact}%%%%%%%%%%%%%%%%%%%%%%%%%%%%%%%%%%%%%%%%%%%%%%
The right-hand side of \eqref{normalWU} is called a {\it Bruce-West normal form} \cite{BW1998}. There are several studies of geometry {of} Whitney {umbrellas} using this form, see \cite{BW1998,FH2012,HHNUY2014,MN2015,T2007,W1995}, for example.

Let $c:(\mathbb{R},0)\to(\mathbb{R}^2,0)$ be a curve. The curve $c$ is said to be of {\it $m$-th multiplicity} if
\[
c(0)=c'(0)=\cdots=c^{(m-1)}(0)=0, c^{(m)}(0)\ne0.
\]
If there exists $m$ such that $c$ is of $m$-th multiplicity, then $c$ is said to be of {\it finite multiplicity}. We assume that $c$ is of finite multiplicity. Then there exists $m$ such that
\begin{equation}\label{c}
c(x) =(c_1(x),c_2(x))x^m, \quad (c_1(0),c_2(0)) \neq(0,0)
\end{equation}
holds.
Let $W:(\mathbb{R}^2,0)\to (\mathbb{R}^3 ,0)$ be a Whitney umbrella and $c_w:(\mathbb{R},0)\to (\mathbb{R}^2,0)$ be a curve of finite multiplicity. We define $\gamma=W\circ c_w$. We can smoothly extend a unit normal vector field of $W$ along $\gamma$  across the Whitney umbrella singularity as follows.

\begin{prop}%%%%%%%%%%%%%%%%%%%%%%%%%%%%
\label{propositionN}
Let $c_w : (\mathbb{R},0) \to (\mathbb{R}^2 ,0)$ be a curve of finite multiplicity, and let $W:(\mathbb{R}^2 ,0) \to (\mathbb{R}^3 ,0)$ be a Whitney umbrella. Then a unit normal vector field $\bm{n}$ of $W$ along $\gamma$ can be smoothly extended across the origin
\end{prop}%%%%%%%%%%%%%%%%%%%%%%%%%%%%%

\begin{proof}%==============================================
We take $W(u,v)$ and $c_w(x)$ given by \eqref{normalWU} and \eqref{c} respectively. Then we have
\begin{align*}
&W_u \times W_v(u,v) = \\
&\left(
\begin{array}{c}
	vA_v(u,v) - uA_u(u,v)- A_u(u,v)B_u(u)+O(u,v)^k \\
	-A_v(u,v) +O(u,v)^k \\
	u +B_u(u) +O(u,v)^k
\end{array}
\right),
\end{align*}
where $( )_u = \partial /\partial u , ( )_v = \partial /\partial v$. Thus we have
\begin{align*}
&(W_u \times W_v) \circ c_w (x) \\
&=\left(
\begin{array}{c}
	(c_1(x)\tilde{A}_u(x)-c_2(x)\tilde{A}_v(x))x^{2m}
		-\tilde{A}_u(x)\tilde{B}_u(x) x^{3m}+O(x)^{mk} \\
	-\tilde{A}_v(x)x^m +O(x)^{mk} \\
	c_1(x)x^m +\tilde{B}_u(x)x^{2m} +O(x)^{mk}
\end{array}
\right) \\
&=
\left(
\begin{array}{c}
	0+O(x)^{2m} \\
	-(a_{11}c_1(x) + a_{02}c_2(x))x^m +O(x)^{2m} \\
	c_1(x)x^m +O(x)^{2m}
\end{array}
\right),
\end{align*}
where 
\begin{align*}
&\tilde{B}_u(x)
=B_u(c_w(x))/x^{2m}
=b_2 c_1(x)^2+O(x)^{m},\\
&\tilde{A}_u(x)
=A_u(c_w(x))/x^m
= (a_{20}c_1(x) + a_{11}c_2(x)) +O(x)^{m},\\
&\tilde{A}_v(x)
=A_v(c_w(x))/x^m
= (a_{11}c_1(x) + a_{02}c_2(x)) +O(x)^{m}.
\end{align*}
We set
\begin{equation}\label{defn}
\mathcal{N}(x) = \dfrac{1}{x^m}(W_u \times W_v) \circ c_w (x).
\end{equation}
Since $a_{02} \neq 0 , (c_1(0) ,c_2(0)) \neq (0,0)$, we have
\[
\mathcal{N}(0) =(0, -(a_{11}c_1(0) + a_{02}c_2(0)), c_1(0))^T \neq (0,0,0)^T. 
\]
Therefore we can take a unit vector field
\begin{equation}\label{n}
\bm{n} (x) = \frac{\mathcal{N}(x)}{|\mathcal{N}(x)|}
\end{equation}
along $\gamma$.
\end{proof}%================================================
Similarly, a unit tangent vector field of $W$ along $\gamma$ can also be smoothly extended across the Whitney umbrella singularity.

\begin{prop}%%%%%%%%%%%%%%%%%%%%%%%%%%%%
Let $c_w : (\mathbb{R},0) \to (\mathbb{R}^2 ,0) $ be a curve of finite multiplicity, and let $W:(\mathbb{R}^2 ,0) \to (\mathbb{R}^3 ,0)$ be a Whitney umbrella. Then a unit tangent vector field of $W$ along $\gamma$ can be smoothly extended across the origin.
\label{propositionE}
\end{prop}%%%%%%%%%%%%%%%%%%%%%%%%%%%%%

\begin{proof}%==============================================
We take $W(u,v)$ and $c_w(x)$ given by \eqref{normalWU} and \eqref{c} respectively. Then, we have
\[
\gamma(x)=W \circ c_w(x) =
\left(
\begin{array}{c}
	c_1(x)x^m \\
	c_1(x)c_2(x)x^{2m}+B(c_w(x))+O(x)^{m(k+1)} \\
	A(c_w(x)) +O(x)^{m(k+1)}
\end{array}
\right).
\]
Differentiating $W \circ c_w (x)$, we have
\begin{align*}
&\gamma'(x) \\
&=\left(
\begin{array}{c}
	mc_1(x)x^{m-1} +{c_1}'(x)x^m \\
	2mc_1(x)c_2(x)x^{2m-1} +O(x)^{2m} \\
	m(a_{20}c_1(x)^2 + 2a_{11}c_1(x)c_2(x)
		+ a_{02}c_2(x)^2)x^{2m-1}+O(x)^{2m}
\end{array}
\right), 
\end{align*}
where $ ' = \partial / \partial x $. We show the proposition {by considering} the following cases.

\begin{enumerate}
	\item[(1)] $c_1(0) \neq 0$,
	\item[(2)] $c_1(0) = {c_1}'(0) = \cdots
		={c_1}^{(l-1)}(0) =0 ,{c_1}^{(l)}(0) \neq 0 ,\ 1 \le l < m$,
	\item[(3)] $c_1(0) = {c_1}'(0) = \cdots
		={c_1}^{(m-1)}(0) =0, {c_1}^{(m)}(0) \neq 0$,
	\item[(4)] $c_1(0) = {c_1}'(0) = \cdots ={c_1}^{(m)}(0) =0$.
\end{enumerate}
If $c_w$ satisfies $(1)$, then we have a map $\mathcal{E}_1 : (\mathbb{R},0) \to (\mathbb{R}^3,0)$ such that
\[
\gamma'(x) =
\mathcal{E}_1(x)x^{m-1}, \quad
\mathcal{E}_1(0) \neq 0,
\]
where 
\[
\mathcal{E}_1(x)=
\left(
\begin{array}{c}
	mc_1(x) +{c_1}'(x)x \\
	2mc_1(x)c_2(x)x^{m} +O(x)^{m+1} \\
	m(a_{20}c_1(x)^2 + 2a_{11}c_1(x)c_2(x)
		+ a_{02}c_2(x)^2)x^{m}+O(x)^{m+1}
\end{array}
\right).
\]
If $c_w$ satisfies $(2)$, then it holds that $c_2(0) \neq 0$, and we set $\bar{c}_1(x)$ satisfying that $c_1(x) =\bar{c}_1(x)x^l , \bar{c}_1(0) \neq 0$. Then we have a map $\mathcal{E}_2 : (\mathbb{R},0) \to (\mathbb{R}^3,0)$ such that
\[
\gamma'(x) =\mathcal{E}_2(x)x^{m+l-1}, \quad \mathcal{E}_2(0) \neq 0,
\]
where
\[
\mathcal{E}_2(x)=
\left(
\begin{array}{c}
	(m+l)\bar{c}_1(x) +{\bar{c}_1}'(x)x \\
	(2m+l)\bar{c}_1(x)c_2(x)x^m +O(x)^{m+1} \\
	ma_{02}c_2(x)^2x^{m-l}+O(x)^{m-l+1}
\end{array}
\right).
\]
If $c_w$ satisfies $(3)$ or $(4)$, then it holds that $c_2(0) \neq 0$, and {we} set $\tilde{c}_1(x)$ satisfying that $c_1(x) = \tilde{c}_1(x) x^m$. Then we have a map $\mathcal{E}_3 : (\mathbb{R},0) \to (\mathbb{R}^3,0)$ such that
\[
\gamma'(x) =\mathcal{E}_3(x)x^{2m-1}, \quad \mathcal{E}_3(0) \neq 0,
\]
where
\[
\mathcal{E}_3(x)=
\left(
\begin{array}{c}
	2m\tilde{c}_1(x) +{\tilde{c}_1}'(x)x \\
	3m\tilde{c}_1(x)c_2(x)x^m+O(x)^{m+1} \\
	m a_{02}c_2(x)^2+O(x)^{1}
\end{array}
\right).
\]
Therefore we can take the unit vector field:
\begin{align}\label{e}
\bm{e}(x) =
\begin{cases}
	\mathcal{E}_1(x)/|\mathcal{E}_1(x)|, \quad \text{for} \ (1), \\
	\mathcal{E}_2(x)/|\mathcal{E}_2(x)|, \quad \text{for} \ (2), \\
	\mathcal{E}_3(x)/|\mathcal{E}_3(x)|,
	\quad \text{for} \ (3) \ \text{and} \ (4)
\end{cases}
\end{align}
along $\gamma$.
\end{proof}%================================================

At $x=0$, the unit tangent vector $\bm{e}$ of $W$ satisfies that
\[
\bm{e}(0) =
\begin{cases}
	(c_1(0), 0, 0)^T/|c_1(0)|, \quad &\text{if} \ (1)\ \text{holds}, \\
	(\bar{c}_1(0), 0, 0)^T/|\bar{c}_1(0)|,\quad &\text{if} \ (2)\ \text{holds}, \\
	(\tilde{c}_1(0), 0, a_{02}c_2(0))^T/
	\sqrt{\tilde{c}_1(0)^2+{a_{02}}^2c_2(0)^2},\quad &\text{if} \ (3)\ \text{holds}, \\
	(0, 0, a_{02}c_2(0))^T/|a_{02}c_2(0)|,\quad &\text{if} \ (4)\ \text{holds}.
\end{cases}
\]
The tangent plane of the Whitney umbrella at the origin degenerates into a line. We call this line a {\it tangent line} {of the Whitney umbrella}. There exists non-zero vector $\eta \in T_0 \mathbb{R}^2$ such that $dW_0(\eta)=0$. We call $\eta$ a {\it null vector} (cf. \cite{KRSUY2005}). 
The plane spanned by the tangent line and $\eta\eta W(0)$ is called the {\it principal plane}, where $\eta\eta W$ is the twice directional derivative of $W$ with respect to $\eta$. Let $L_t:(\mathbb{R},0) \to \mathbb{R}^3$ be the tangent line of $W$. The line $L_i:(\mathbb{R},0) \to \mathbb{R}^3$ is called the {\it principal intersection line} of $W$ if the line is the intersection of the principal plane and the normal plane {(see Figure \ref{plane})}. 

\begin{ex}
We set $W_1(u,v)=(u,v^2,uv)^T$ which is a Whitney umbrella. Then we obtain $L_t(u) = (u,0,0)^T, L_i(u)=(0,u,0)^T$. 

We set
\[
c_{w1}(x)=\left(x^2,\frac{x^5}{5}\right),\quad c_{w2}(x)=\left(x^2,x\right),\quad c_{w3}(x)=\left(\frac{x^5}{5},x^2\right).
\]
We consider the curves $\gamma_1(x)=W_1\circ c_{w1}(x)$, $\gamma_2(x)=W_1\circ c_{w2}(x)$, and $\gamma_3(x)=W_1\circ c_{w3}(x)$. Then the curve $\gamma_1$ (respectively, $\gamma_2$, $\gamma_3$) satisfies the condition $(1)$ (respectively, $(3)$, $(4)$).
The curves $\gamma_1$, $\gamma_2$, and $\gamma_3$ are tangent to {the} principal plane at $x=0$ ({see} Figure \ref{cur}). In particular, $\gamma_1$ is tangent to the tangent line of $W_1$ and $\gamma_3$ is tangent to the principal intersection line of $W_1$.
We remark that the direction of the unit tangent vector changes between (2) and (4) via (3).
\end{ex}

%%%%%%%%%%%%%%%%%%%%%%%%%%%%%%%%%%%%%%%%%%%%%%%%%%%%%%%%%%%%%%%%%%%%%%%%%%
\begin{figure}[ht]
\centering
\includegraphics[width=0.8\textwidth]{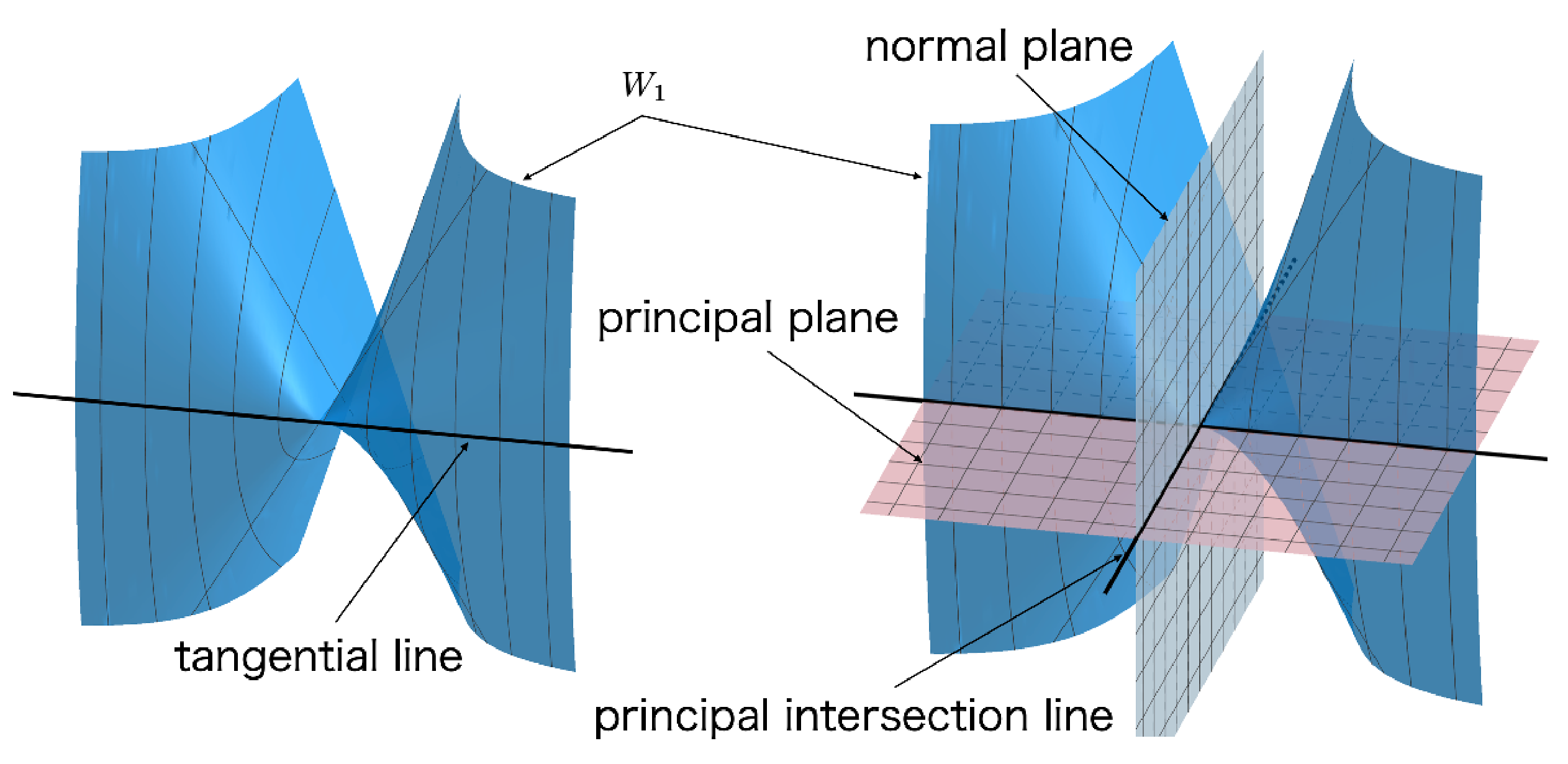}
\caption{The tangent line, principal plane, normal plane, and principal intersection line of $W_1$}
\label{plane}
\end{figure}
%%%%%%%%%%%%%%%%%%%%%%%%%%%%%%%%%%%%%%%%%%%%%%%%%%%%%%%%%%%%%%%%%%%%%%%%%%

%%%%%%%%%%%%%%%%%%%%%%%%%%%%%%%%%%%%%%%%%%%%%%%%%%%%%%%%%%%%%%%%%%%%%%%%%%
\begin{figure}[ht]
\centering
\includegraphics[width=0.8\textwidth]{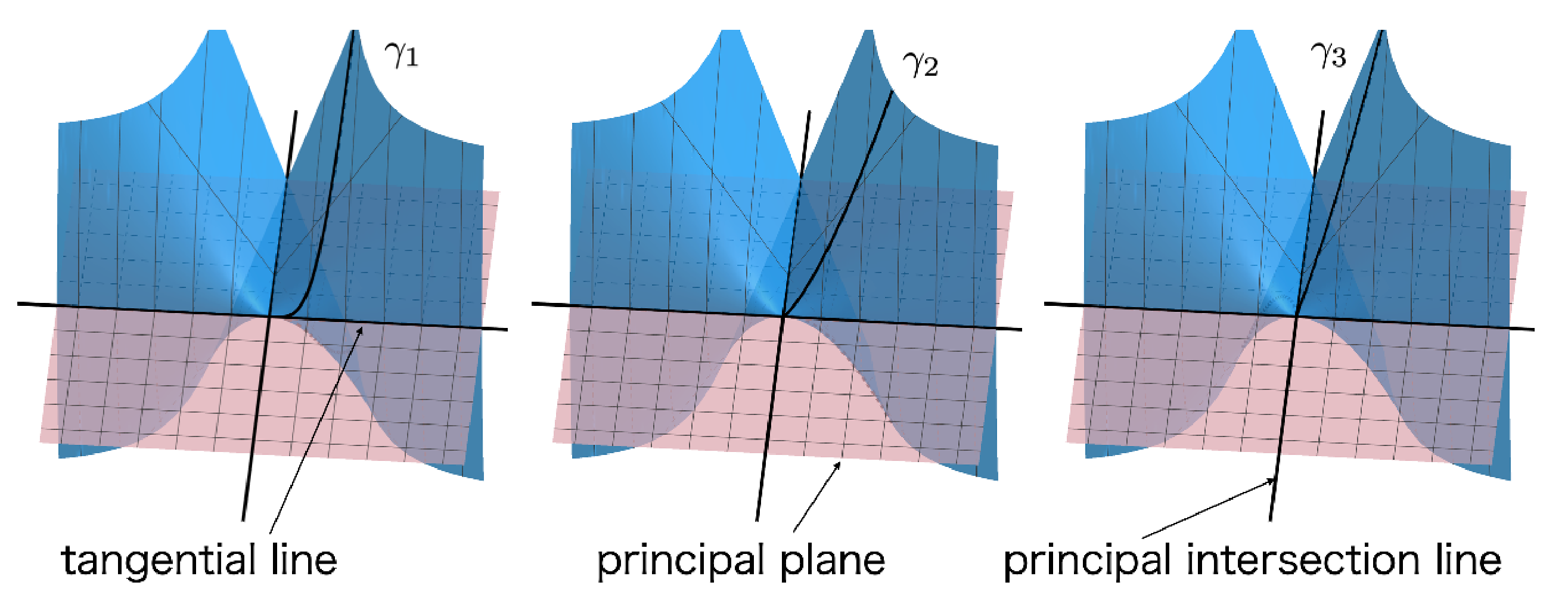}
\caption{The curves $\gamma_1$ (left), $\gamma_2$ (center), and $\gamma_3$ (right)}
\label{cur}
\end{figure}
%%%%%%%%%%%%%%%%%%%%%%%%%%%%%%%%%%%%%%%%%%%%%%%%%%%%%%%%%%%%%%%%%%%%%%%%%%

\section{Geometry {of} curves passing through Whitney umbrella}%%%%%%%%%%%%%%%%%%%%%%%%%%%%%%%%%%%%%%%%%%
\subsection{Darboux frame and Curvatures}%%%%%%%%%%%%%%%%%%%%%%%%%%%%%%%%%%%%%%%%%%%%%%%%%%%%%
Let $W$ and $c_w$ be \eqref{normalWU} and \eqref{c} respectively. Using $(\ref{n})$ and $(\ref{e})$, we obtain $\bm{n}(x)$ and $\bm{e}(x)$
which are the normal vector and the tangent vectors of $W$ along $\gamma$. We define $\bm{b}(x) := \bm{n}(x) \times \bm{e}(x)$. Then we obtain a Darboux frame 
\begin{equation}\label{ebn}
\{ \bm{e}(x) ,\bm{b}(x) ,\bm{n}(x) \}
\end{equation}
along $\gamma$. Using this frame, we define functions $\kappa_1 , \kappa_2 , \kappa_3 : (\mathbb{R} ,0) \to \mathbb{R}$ by the following Frenet-Serret type formula:
{
\begin{align}\label{frenet}
\left(
	\bm{e}'(x), 
	\bm{b}'(x), 
	\bm{n}'(x)
\right)^T
=
\left(
\begin{array}{ccc}
	0 & \kappa _1 (x) & \kappa _2 (x) \\
	-\kappa _1 (x) & 0 & \kappa _3 (x) \\
	-\kappa _2 (x) & -\kappa _3 (x) & 0
\end{array}
\right)
\left(
	\bm{e} (x),
	\bm{b} (x),
	\bm{n} (x)
\right)^T.
\end{align}
}
Note that the functions $\kappa_1, \kappa_2$, and $\kappa_3$ depend on the parameter $x$. We have the following relations between $\kappa_1, \kappa_2, \kappa_3$ and the geodesic curvature $\kappa_g$, the normal curvature $\kappa_{\nu}$, and the geodesic torsion $\kappa_t$, which are defined on the set of regular points of $W$. By Propositions $\ref{propositionN}$ and $\ref{propositionE}$, there exist $\alpha$, $\beta$ {so that} we can set $\mathcal{E}$ and $\mathcal{N}$ by
\[
\gamma '(x) = \mathcal{E}(x) x^{\alpha} ,\mathcal{E}(0) \neq 0,
\]
\[
(W_u \times W_v) \circ c_w (x) =
\mathcal{N}(x) x^{\beta} ,\mathcal{N}(0) \neq 0.
\]
\begin{lem}\label{relation-k}%%%%%%%%%%%%%%%%%%%%%%
Under the above assumptions on $\alpha,\beta$, $\mathcal{E}$ and $\mathcal{N}$, the following relations hold
\begin{flalign}
\label{k1kg}
&\kappa_g(x) =
\sgn{(x^{\alpha +\beta})}\frac{\kappa_1(x)}{|\mathcal{E}(x)|x^{\alpha}}, \\
\label{k2kn}
&\kappa_{\nu}(x) =
\sgn{(x^{\beta})}\frac{\kappa_2(x)}{|\mathcal{E}(x)|x^{\alpha}}, \\
\label{k3kt}
&\kappa_t(x) =
\frac{\kappa_3(x)}{|\mathcal{E}(x)|x^{\alpha}}.
\end{flalign}
\end{lem}%%%%%%%%%%%%%%%%%%%%%%%%%%%%%%%

\begin{proof}%==============================================
We set
\[
\bar{\bm{e}}(x) =
\frac{\mathcal{E}(x) x^{\alpha}}{|\mathcal{E}(x) x^{\alpha}|},\quad
\bar{\bm{b}}(x) =
\bar{\bm{n}}(x) \times \bar{\bm{e}}(x),\quad
\bar{\bm{n}}(x) =
\frac{\mathcal{N}(x) x^{\beta}}{|\mathcal{N}(x) x^{\beta}|}.
\]
The frame $\{\bar{\bm{e}},\bar{\bm{b}},\bar{\bm{n}}\}$ is a Darboux frame of the curve $\gamma$ at regular points. Note that $\{ \bar{\bm{e}}(x), \bar{\bm{b}}(x), \bar{\bm{n}}(x)\} $ are not defined at $x = 0$. In this case, the geodesic curvature is given by:
\[
\kappa_g(x) =
\frac{\langle \gamma''(x), \bar{\bm{b}}(x) \rangle}{|\gamma'(c)|^2} 
\]
where $\langle \ , \ \rangle$ is {the} inner product. A straightforward calculation shows that
\begin{align}\label{cal_kg}
\kappa_g(x) =
\frac{\sgn{(x^{2\alpha})}\sgn{(x^{\beta})}}{|\mathcal{E}(x)x^{\alpha}|}
\left\langle
\frac{\mathcal{E}'(x)}{|\mathcal{E}(x)|},
\frac{\mathcal{N}(x)}{|\mathcal{N}(x)|} \times
\frac{\mathcal{E}(x)}{|\mathcal{E}(x)|}
\right\rangle.
\end{align}
On the other hand, we obtain
\begin{align}\label{cal_k1}
\kappa_1 (x) =
\langle \bm{e}'(x), \bm{b}(x) \rangle =
\left\langle
\frac{\mathcal{E}'(x)}{|\mathcal{E}(x)|},
\frac{\mathcal{N}(x)}{|\mathcal{N}(x)|} \times
\frac{\mathcal{E}(x)}{|\mathcal{E}(x)|}
\right\rangle.
\end{align}
By equations $(\ref{cal_kg})$ and $(\ref{cal_k1})$, we obtain $(\ref{k1kg})$. By the same method, we obtain $(\ref{k2kn})$ and $(\ref{k3kt})$.
\end{proof}%===================================================================

\subsection{Degrees and top-terms of curvatures}%%%%%%%%%%%%%%%%%%%%%%%%%%%%%%%%%%%%%%%%%%%%%%%%%%
Let $W$ and $c_w$ be \eqref{normalWU} and \eqref{c} respectively. Then the limit vector
\begin{equation}\label{limc}
\lim_{x\to0}\frac{{c_w}'(x)}{|{c_w}'(x)|}
\end{equation}
is well-defined. If $c_2(0)=0$, then the limit vector does not generate  $\Ker dW_0=\{(u,v)\in(\mathbb{R}^2,0)|u=0\}$. Then $\gamma$ is tangent to the tangential line of $W$. Therefore, the local properties of $\gamma$ are the same as the tangential line of $W$.
On the other hand, if $c_1(0)=0$, then we see the limit vector \eqref{limc} generates $\Ker dW_0$. In this case, $\gamma$ may not be tangent to the tangential line of $W$.  Thus we assume that $c_1(0)=0$. By changing parameterization, we may assume
\begin{equation}\label{c_w}
c_w(x)=
\begin{cases}
\displaystyle\left(x^{mp+q}\sum_{i=0}c_{i}x^{i},x^m\right) &(\text{if}\ p\geq1,\ 1\le q<m),\\
\displaystyle\left(x^{mp}\sum_{i=0}c_{i}x^{i},x^m\right) &(\text{if}\ p\geq2,\ q=0),
\end{cases}
\end{equation}
where $c_{i}\in\mathbb{R}$. The coefficient $c_{0}$ is related to the bias of a cusp \cite{HS2019}.

If the degree of $\kappa_1$ (respectively, $\kappa_2$ or $\kappa_3$) with respect to $x$ is equal or greater than $\alpha$, then the curvature $\kappa_g$ (respectively, $\kappa_{\nu}$ or $\kappa_t$) can be smoothly extended across the singularity of $W$. We study degrees and top-terms of $\kappa_1,\ \kappa_2, \text{and} \ \kappa_3$.

We assume that $p\geq1 , 1 \le q < m$, and $c_w(x)$ satisfies $(\ref{c_w})$. By \eqref{frenet}, it holds that
\begin{align}
&\kappa _1 (x) =
\begin{cases}
	\tilde{\kappa} _1(x) x^{m-q-1} & (\text{if}\ p=1), \\
	\tilde{\kappa} _1(x) x^{m(p-2) +q -1} & (\text{if}\ p=2,3), \\
	\tilde{\kappa} _1(x) x^{2m-1} & (\text{if}\ p\geq4),
\end{cases}\nonumber \\
&\kappa _2 (x) = \tilde{\kappa} _2(x) x^{m-1} \nonumber \\
&\kappa _3 (x) =
\begin{cases}
	\tilde{\kappa} _3(x) x^{q-1} & (\text{if}\ p=1), \\
	\tilde{\kappa} _3(x) x^{m-1} & (\text{if}\ p\geq2 ) .
\end{cases}\nonumber 
\end{align}
where $\tilde{\kappa}_1, \tilde{\kappa}_2, \tilde{\kappa_3}:(\mathbb{R},0) \to \mathbb{R}$ are $C^{\infty}$-functions, and $\tilde{\kappa}_1(0)$ (respectively, $\tilde{\kappa}_2(0)$ or $\ \tilde{\kappa}_3(0))$ is the top-term of $\kappa_1$ (respectively, $\kappa_2$ or $\kappa_3$). Calculating $\tilde{\kappa}_1(0)$, $\tilde{\kappa}_2(0)$, and $\tilde{\kappa}_3(0)$, we have
\begin{align}
&\tilde{\kappa}_1(0) =
\frac{1}{|\mathcal{E}(0)|^2|\mathcal{N}(0)|} 
\begin{cases}
	m({m}^2-q^2){a_{02}}^2c_0 &(\text{if}\ p=1), \\
	-m\{ m(p-2)+q\} (mp+q){a_{02}}^2c_0 &(\text{if}\ p=2,3), \\
	-{m}^2{a_{02}}^2b_3 &(\text{if}\ p\geq4),
\end{cases} 
\label{k11}\\
&\tilde{\kappa}_2(0) =
\frac{1}{|\mathcal{E}(0)||\mathcal{N}(0)|} 
\begin{cases}
	-m(m+2q) a_{02}c_0 &(\text{if}\ p=1), \\
	-{m}^2a_{02}b_3 /2 &(\text{if}\ p\geq2),
\end{cases} 
\label{k21}\\
&\tilde{\kappa}_3(0) =
\frac{1}{|\mathcal{E}(0)||\mathcal{N}(0)|^2} 
\begin{cases}
	-q(mp+q)a_{02}c_0 &(\text{if}\ p=1), \\
	m^2{a_{02}}^3 &(\text{if}\ p\geq2).
\end{cases}
\label{k31}
\end{align}

We assume that $p\geq2, q=0$, and $c_w(x)$ satisfies $(\ref{c_w})$. By the same calculation {as} above, we have
\begin{align}
&\kappa _1 (x) =
\begin{cases}
	\tilde{\kappa} _1(x) x^{m-1} & (\text{if}\ p =1,2,3), \\
	\tilde{\kappa} _1(x) x^{2m-1} & (\text{if}\ p\geq 4),
\end{cases}\nonumber \\
&\kappa _2 (x) = \tilde{\kappa} _2(x) x^{m-1},\nonumber \\
&\kappa _3 (x) = \tilde{\kappa} _3(x) x^{m-1},\nonumber 
\end{align}
and
\begin{align}
&\tilde{\kappa}_1(0) =
\frac{1}{|\mathcal{E}(0)|^2|\mathcal{N}(0)|} 
\begin{cases}
	a_{02}(6a_{11}{c_0}^2+a_{03}c_0-3a_{02}c_1) &(\text{if}\ p=2,m=1), \\
	m^3 a_{02} c_0(6 a_{11} c_0 + a_{03}) &(\text{if}\ p=2, m\geq2), \\
	-3{m}^3{a_{02}}^2c_0 &(\text{if}\ p=3), \\
	-{m}^3{a_{02}}^2(8c_0 +b_3 /2) &(\text{if}\ p=4), \\
	{m}^3{a_{02}}^2b_3 /2 &(\text{if}\ p\geq 5),
\end{cases}
\label{k12} \\
&\tilde{\kappa}_2(0) =
\frac{1}{|\mathcal{E}(0)||\mathcal{N}(0)|} 
\begin{cases}
	-{m} ^2 a_{02} (3c_0 +b_3 /2 ) &(\text{if}\ p=2), \\
	-{m}^2(p-1)a_{02}b_3 /2 &(\text{if}\ p\geq3),
\end{cases} 
\label{k22}\\
&\tilde{\kappa}_3(0) =
\frac{1}{|\mathcal{E}(0)||\mathcal{N}(0)|^2} 
\begin{cases}
	-{m} ^2 a_{02}(2{c_0}^2+b_3c_0-{a_{02}}^2) &(\text{if}\ p=2), \\
	{m}^2{a_{02}}^3 &(\text{if}\ p\geq3).
\end{cases}
\label{k32}
\end{align}

{In} \eqref{k22} and \eqref{k32}, if $3c_0+b_3/2=0$, then it holds
\[
2c_0^2+b_3c_0-a_{02}^2=-4c_0^2-a_{02}^2\ne0.
\]
Then we have the following lemma.

\begin{lem}\label{k2k3}
We set $W$ and $c_w$ by \eqref{normalWU} and \eqref{c_w}. Then the pair $(\kappa_2, \kappa_3)$ is of finite multiplicity.
\end{lem}

\subsection{Differential geometric meanings of top-terms of curvatures}%%%%%%%%%%%%%%%%%%%%%%%%%%%%%%%%%%%%%%%
In this section, we study the geometric meaning of the vanishing of the top-terms of $\kappa_{1}$, $\kappa_2$, and $\kappa_3$. The coefficients $c_i$ are differential geometric invariants of $c_w$. The coefficients $a_{ij}$ and $b_i$ are differential geometric invariants of $W$.
For example, we observe the top-term \eqref{k12} of $\kappa_1$ in this case $p=2,m=1$. The top-term $\tilde\kappa_1(0)$ vanishes if and only if 
\begin{equation}\label{eqk1}
6a_{11}{c_0}^2+a_{03}c_0-3a_{02}c_1=0.
\end{equation}
If $\gamma=W\circ c_w$ satisfies the equation \eqref{eqk1}, then the degrees of divergence of $\kappa_1$ is greater than or equal to $1$. We remark that this property does not depend on the parameter $x$.
The equation \eqref{eqk1} is the relational expression between differential geometric invariants of $W$ and differential geometric invariants of $c_w$.
In other words, equation \eqref{eqk1} has a geometric meaning with respect to $W$ and $c_w$.

For \eqref{e}, if $(1)$ and $(2)$ hold, then the vector $\mathcal{E}(x)$ of $\gamma$ is tangent to the tangent line of the Whitney umbrella. In this case, vanishing of the top-terms of $\kappa_1, \kappa_2, \kappa_3$ is determined by the information of $\gamma$ itself. However, if $(3)$ and $(4)$ hold, then $\mathcal{E}(x)$ is not tangent to that of Whitney umbrella. 
We set $c_w$ as
\begin{align}\label{c_2m}
c_w(x) = (x^{2m}\sum_{i=0}c_ix^i, x^m),
\end{align}
where $c_i\in\mathbb{R}$. In this case, the contributions from the Whitney umbrella and from the curve are of the same order, and we study how both sets of information contribute to the functions. We obtain the top-term $\tilde{\kappa}_1(0)$ (respectively, $\tilde{\kappa}_2(0)$ or $\tilde{\kappa}_3(0)$) of $\kappa_1$ (respectively, $\kappa_2$ or $\kappa_3$):
\[
\tilde{\kappa}_1(0) =
\frac{m^3 a_{02}A}{|\mathcal{E}(x)|^2|\mathcal{N}(x)|},\ 
\tilde{\kappa}_2(0) =
-\frac{{m} ^2 a_{02}B}{|\mathcal{E}(x)||\mathcal{N}(x)|},\ 
\tilde{\kappa}_3(0) =
-\frac{{m} ^2 a_{02}C}{|\mathcal{E}(x)||\mathcal{N}(x)|^2},
\]
where
\begin{equation}\label{rela_k1}
A = 
\begin{cases}
6a_{11} c_0 ^2 +a_{03}c_0 -3a_{02} c_1\ &(\text{if}\ m=1),\\
6a_{11}c_0+a_{03} &(\text{if}\ m\geq 2),
\end{cases}\
B = 3c_0 +\frac{b_3}{2},\
C = 2c_0 ^2 + b_3 c_0 - a_{02} ^2,
\end{equation}
are top-term respectively.

{We set} $W$ and $c_w$ by \eqref{normalWU} and \eqref{c_2m}. Using the Darboux frame \eqref{ebn} at $x=0$, it holds that
\begin{equation}\label{ebn0}
\bm{e}(0) = \frac{(2c_0, 0, a_{02})^T}{\sqrt{4{c_0}^2+{a_{02}}^2}},\ 
\bm{b}(0) = \frac{(-a_{02}, 0, 2c_0)^T}{\sqrt{4{c_0}^2+{a_{02}}^2}},\ 
\bm{n}(0) =\frac{(0, -a_{02}, 0)^T}{|a_{02}|}.
\end{equation}
Let $\pi :\mathbb{R}^3 \to \Pi$ be the projection onto the vector $\bm{e}(0)$, where $\Pi\subset\mathbb{R}^3$ is the plane spanned by $\bm{b}(0)$ and $\bm{n}(0)$. Then we obtain the following theorem.

\begin{thm}%%%%%%%%%%%%%%%%%%%%%%%%%%%%
We assume that $(A,B)\ne(0,0)$. Then the projection $\pi \circ \gamma$ is tangent to $\bm{n}(0)$ at $x=0$ if and only if $A=0$.
The projection $\pi \circ \gamma$ is tangent to $\bm{b}(0)$ at $x=0$ if and only if $B=0$.
\end{thm}%%%%%%%%%%%%%%%%%%%%%%%%%%%%%%
\begin{proof}
Under the notations in \eqref{normalWU} and \eqref{c_2m}. It holds that
\begin{align}
(\pi \circ \gamma)(x) =
&\frac{1}{3\sqrt{4{c_0}^2+{a_{02}}^2}}
\{Ax^{3m}+O(x)^{3m+1}\} \bm{b}(0) \nonumber \\
&-\frac{a_{02}}{|a_{02}|}
\{Bx^{3m}+O(x)^{3m+1}\} \bm{n}(0).\nonumber
\end{align}
If $A=0$, then
\[
(\pi \circ \gamma)(x) =
O(x)^{3m+1}\bm{b}(0)
-\frac{a_{02}}{|a_{02}|}(Bx^{3m} +O(x)^{3m+1})\bm{n}(0).
\]
Differentiating $\pi \circ \gamma$, it holds that
\[
(\pi \circ \gamma)'(x) =
O(x)^{3m}\bm{b}(0)
-\frac{3ma_{02}}{|a_{02}|}(Bx^{3m-1} +O(x)^{3m})\bm{n}(0).
\]
We set $E_{\pi} (x) =(\pi \circ \gamma)'(x) / x^{3m-1}$.
Then it holds that
\[
E_{\pi} (0) = -\frac{3ma_{02}B}{|a_{02}|} \bm{n}(0).
\]
This equation implies $(\pi \circ \gamma)(x)$ is tangent to $\bm{n}(0)$ at the origin. If $A \neq 0$ and $B =0$, we see the assertion by the same method.
\end{proof}

\begin{thm}
The curve passing through the Whitney umbrella singularity is tangent to its self-intersecting curve at $x=0$ if and only if $B=0$.
\end{thm}
\begin{proof}
The Whitney umbrella $W$ has the self-intersecting curve passing through the singularity of $W$. First, we shall give a curve which approximates the self-intersecting curve of Whitney umbrella. Let $d(x) :(\mathbb{R} ,0) \to (\mathbb{R}^2 ,0)$ be a curve:
\[
d(x) =
(d_{11}x + d_{12}x^2 + O(x)^3, d_{21}x + d_{22}x^2 + O(x)^3),\quad
d_{11}, d_{12}, d_{21}, d_{22} \in \mathbb{R},
\]
where we assume that $W \circ d(x)$ is the self-intersecting curve of Whitney umbrella.

Here, we consider $W_0(u,v) = (u,uv,v^2)^T$ and $d_0(x) = (0, x)$. Then the self-intersecting curve of $W_0$ is given $W_0 \circ d_0(x) = (0, 0, x^2)^T$. By Fact $\ref{Fact}$, there exist a rotation map $T :(\mathbb{R}^3, 0) \to (\mathbb{R}^3, 0)$ and a diffeomorphism map $\phi :(\mathbb{R}^2 ,0) \to (\mathbb{R}^2, 0)$ such that
\[
W \circ d(x) = T \circ (W_0 \circ d_0 ) \circ \phi (x).
\]
Since ${d_0}'(x) = (0,1) \neq (0,0)$, it holds that $d'(x) = (d_0 \circ \phi)'(x) \neq (0,0)$. Therefore, $(d_{11}, d_{21}) \neq (0,0)$. Since $W \circ d(x)$ is the self-intersecting curve of Whitney umbrella, it satisfies 
\begin{equation}\label{self_rela}
W \circ d(x) = W \circ d(-x).
\end{equation}
Calculating \eqref{self_rela}, we obtain that
\begin{equation}\label{equ_self}
\left(
\begin{array}{c}
	2d_{11}x + O(x)^3 \\
	\displaystyle 2\left(d_{11}d_{22}+d_{12}d_{21}
		+\frac{b_3}{6}{d_{21}}^3\right)x^3 + O(x)^4 \\
	2A_3 x^3 + O(x)^{4}
\end{array}
\right)
=
\left(
\begin{array}{c}
	0 \\
	0 \\
	0
\end{array}
\right).
\end{equation}
where
\begin{align*}
A_3 =
&a_{20}d_{11}d_{12}+a_{11}(d_{11}d_{22}+d_{12}d_{21})
+a_{02}d_{21}d_{22} \\
&+\frac{a_{30}}{6}{d_{11}}^3 +\frac{a_{21}}{2}{d_{11}}^2d_{21}
+\frac{a_{12}}{2}d_{11}{d_{21}}^2+\frac{a_{03}}{6}{d_{21}}^3.
\end{align*}
By the equation \eqref{equ_self}, it holds that $d_{11} =0$. Then since $(d_{11}, d_{21}) \neq (0,0)$, we have $d_{21} \neq 0$. By changing parameterization, we we may assume that $d_{21} = 1$. By the equation \eqref{equ_self}, we have
\[
d_{12} = -\frac{b_3}{6},\quad
d_{22} = \frac{b_3a_{11} -a_{03}}{6a_{02}}.
\]
Then 
\begin{equation}\label{Wd'}
(W \circ d)'(x) =
\left(
\begin{array}{c}
	\displaystyle-\frac{b_3}{3}x +O(x)^2 \\
	0+O(x)^2 \\
	a_{02}x +O(x)^2
\end{array}
\right)
=
\left(
\begin{array}{c}
	\displaystyle-\frac{b_3}{3} +O(x)^1 \\
	0+O(x)^1 \\
	a_{02} +O(x)^1
\end{array}
\right)x.
\end{equation}
And, considering the curve $c_w(x)$ by \eqref{c_2m}, it holds that
\begin{equation}\label{Wc'}
\gamma'(x) =
\left(
\begin{array}{c}
	2mc_0x^{2m-1} +O(x)^{2m} \\
	0 +O(x)^{2m} \\
	ma_{02}x^{2m-1} +O(x)^{2m}
\end{array}
\right)
= m
\left(
\begin{array}{c}
	2c_0 +O(x)^{1} \\
	0 +O(x)^{1} \\
	a_{02}+O(x)^{1}
\end{array}
\right)x^{2m-1}.
\end{equation}
By \eqref{Wd'} and \eqref{Wc'}, if $B= 3c_0 + b_3 /2 =0$, then the curve $\gamma$ is tangent to the self-intersecting curve of $W$ at the origin.
\end{proof}

Let $f:(\mathbb{R}^2,0)\to(\mathbb{R}^3,0)$ be a smooth map and let $\gamma_c:(\mathbb{R},0)\to(\mathbb{R}^3,0)$ be a curve on $f$. We call $\gamma_c$ the {\it contour generator of $f$ by the direction $\bm{v}$} if the unit normal vector field $\bm{n}_c$ of $f$ along $\gamma_c$ satisfies that $\langle \bm{n}_c,\bm{v}\rangle=0$. We consider the $k$-jet germ $j^{(k)}\bm{n}_c(0)$. We call $\gamma_c$ the {\it $k$-approximational contour generator of $f$ by the direction $\bm{v}$} if it holds {that} $\langle j^{(k)}\bm{n}_c(0),\bm{v}\rangle=0$.

\begin{thm}%%%%%%%%%%%%%%%%%%%%%%%%%%%%%
The curve $\gamma$ is {an} $m$-approximational contour generator of $W$ by the direction $\bm{b}(0)$ if and only if $C=0$.
\end{thm}%%%%%%%%%%%%%%%%%%%%%%%%%%%%%%
\begin{proof}
$W$ and $c_w$ are given by \eqref{normalWU} and \eqref{c_2m}. By \eqref{defn}, we can take a unit normal vector $\bm{n} (x)$. If $W \circ c_w$ is a contour generator of $W$ by the direction $\bm{b}(0)$, then it holds that
\[
\langle \bm{n}(x), \bm{b}(0) \rangle= 0.
\]
Then we obtain that
\[
\langle \bm{n}(x), \bm{b}(0) \rangle =
\frac{1}{\sqrt{4{c_0}^2+{a_{02}}^2}}
(2{c_0}^2 +b_3c_0-{a_{02}}^2)x^m+O(x)^{m+1} =0.
\]
If $C=2{c_0}^2 +b_3c_0-{a_{02}}^2=0$, then it holds that $\langle j^{(m)}\bm{n}(x),\bm{b}(0) \rangle=0$.
\end{proof}

\begin{ex}
We set 
{
\[
W_2(u,v)=\left(u,uv+\frac{v^3}{6},-\frac{v^2}{2}\right)^T,\ 
c_{w4}(x)=\left(-\frac{x^2}{6}+\frac{x^3}{6},x\right),\
c_{w5}(x)=\left(\frac{x^2}{2}+\frac{x^3}{6},x\right).
\]
}
We consider the curve $\gamma_4=W_2\circ c_{w4}$ and $\gamma_5=W_2\circ c_{w5}$. Then $\gamma_4$ (respectively, $\gamma_5$) satisfies $B=0$ (respectively, $\gamma_5$).
Figure \ref{k123}(left) is $W_2$ and $\gamma_4$. The curve $\gamma_4$ is tangent the self-intersecting curve of $W_2$. Figure \ref{k123}(center and right) are $W_2$ and $\gamma_5$. Figure \ref{k123}(right) is a projection of $W_2$ onto the vector $\bm{b}(0)$.
\end{ex}
%%%%%%%%%%%%%%%%%%%%%%%%%%%%%%%%%%%%%%%%%%%%%%%%%%%%%%%%%%%%%%%%%%%%%%%%%%
\begin{figure}
\centering
\includegraphics[width=0.8\textwidth]{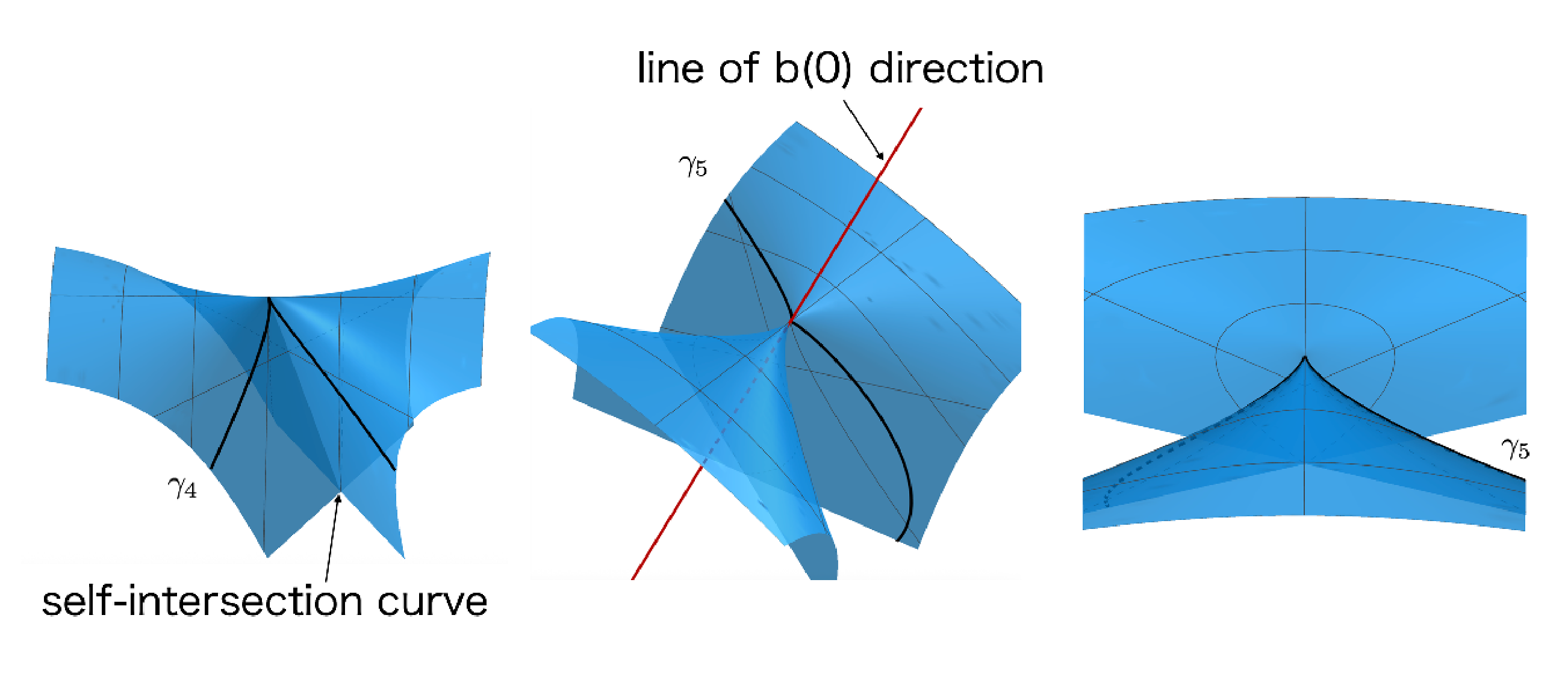}
\caption{Examples of $B=0$ or $C=0$}
\label{k123}
\end{figure}
%%%%%%%%%%%%%%%%%%%%%%%%%%%%%%%%%%%%%%%%%%%%%%%%%%%%%%%%%%%%%%%%%%%%%%%%%

\section{Developable surfaces along a curve passing through Whitney umbrella}%%%%%%%%%%%%%%%%%%%%%%%%%%%%%%%%%%%%%%%%%%

Developable surfaces are classified into cylinders, cones, tangent developable surfaces, and {surfaces obtained by gluing them}. A cylinder is defined by the condition that the direction of the director curve is constant, and a non-cylindrical developable surface is defined by the condition that the derivative of the director curve does not vanish. A cone is defined by the condition that the direction of the striction curve is constant.

We introduce a pseudo-cylindrical developable surface as {a ruled surface} in which the derivative of the director curve has a finite order zero, and similarly define a pseudo-conical developable surface as {a ruled surface} in which the derivative of the striction curve has a finite order zero. A pseudo-cylindrical developable surface can generate developable surfaces that are progressively closer to cylinders, and a pseudo-conical surface can generate surfaces that approach cones.

\subsection{Ruled surfaces and developable surfaces}%%%%%%%%%%%%%%%%%%%%%%%%%%%%%%%%%%%%%%%%%%%%%%%

Let $\boldsymbol\gamma : (\mathbb{R} ,0) \to (\mathbb{R}^3 ,0)$ and $\boldsymbol\xi : (\mathbb{R}, 0) \to \mathbb{R}^3 \setminus \{ 0\}$ be curve-germs. The map $F_{(\boldsymbol\gamma, \boldsymbol\xi)}$ defined by
\begin{equation}\label{ruled}
F_{(\boldsymbol\gamma, \boldsymbol\xi)}(x,y) =
\boldsymbol\gamma (x) + y \boldsymbol\xi (x),
\end{equation}
is called a {\it ruled surface}. We call $\boldsymbol\gamma$ a {\it base curve} of $F_{(\boldsymbol\gamma, \boldsymbol\xi)}$ and $\boldsymbol\xi$ a {\it director curve} of $F_{(\boldsymbol\gamma, \boldsymbol\xi)}$. For a fixed $x_0 \in (\mathbb{R},0)$, the line defined by $\boldsymbol\gamma (x_0) + y \boldsymbol\xi (x_0)$ is called a {\it ruling} of $F_{(\boldsymbol\gamma, \boldsymbol\xi)}$. Using the notation $\bar{\boldsymbol\xi}(x)=\boldsymbol\xi(x)/|\boldsymbol\xi(x)|$, we have $\image F_{(\boldsymbol\gamma, \boldsymbol\xi)} =\image F_{(\boldsymbol\gamma, \bar{\boldsymbol\xi})}$. Therefore, without loss of generality, we may assume that $|\boldsymbol\xi(x)|=1$. A ruled surface is said to be {\it developable} if the Gaussian curvature vanishes on the regular part. It is known (cf., \cite{IT2003}) that a ruled surface $F_{(\boldsymbol\gamma, \boldsymbol\xi)}$ by \eqref{ruled} is developable if and only if
 \begin{equation}\label{dev_rela}
\det(\boldsymbol\gamma'(x), \boldsymbol\xi (x), \boldsymbol\xi '(x)) \equiv 0,
\end{equation}
where $\equiv$ means that equality holds identically. Let $F_{(\boldsymbol\gamma, \boldsymbol\xi)}$ be a ruled surface. If the direction of the director curve $\boldsymbol\xi$ is constant, we call $F_{(\boldsymbol\gamma, \boldsymbol\xi)}$ a {\it cylinder}. Then $F_{(\boldsymbol\gamma, \boldsymbol\xi)}$ is a cylinder if and only if $\boldsymbol\xi '(x) \equiv 0$. If $\boldsymbol\xi '(x) \neq 0$ for any $x \in (\mathbb{R},0)$, $F_{(\boldsymbol\gamma, \boldsymbol\xi)}$ is said to be {\it non-cylindrical}. The ruled surface $F_{(\boldsymbol\gamma, \boldsymbol\xi)}$ is said to be {\it $k$-th pseudo-cylindrical} at the origin if there exist $g:(\mathbb{R},0)\to \mathbb{R}^3$ and  $k \in \mathbb{Z}_{\geq 0}$ such that
\[
\boldsymbol\xi'(x) = g(x)x^k \quad g(0) \neq 0.
\]
By definition, $F_{(\boldsymbol\gamma, \boldsymbol\xi)}$ is $0$-th pseudo-cylindrical if and only if it is non-cylindrical. A map $\bm{s}(x) :(\mathbb{R}, 0) \to (\mathbb{R}^3, 0)$:
\[
\bm{s}(x) =
F_{(\boldsymbol\gamma, \boldsymbol\xi)}(x, y(x)) =
\boldsymbol\gamma (x) + y(x) \boldsymbol\xi (x)
\]
is called a {\it striction curve} if $\bm{s}(x)$ satisfies
\begin{equation}\label{striction_eq}
\langle \bm{s}'(x), \boldsymbol\xi'(x) \rangle\equiv0.
\end{equation}
If $F_{(\boldsymbol\gamma, \boldsymbol\xi)}$ is non-cylindrical, then the striction curve of $F_{(\boldsymbol\gamma, \boldsymbol\xi)}$ is given by
\[
\bm{s}(x) =
\boldsymbol\gamma (x) -
\frac{\langle \boldsymbol\gamma'(x),\boldsymbol\xi'(x) \rangle}
{\langle \boldsymbol\xi'(x), \boldsymbol\xi'(x) \rangle}
\bar{\boldsymbol\xi}(x).
\]
It is known that a singular value of the non-cylindrical ruled surface is located on the striction curve. We say that $F_{(\boldsymbol\gamma, \boldsymbol\xi)}$ is a {\it cone} if it holds that $\bm{s}'(x) \equiv0$. A non-cylindrical ruled surface $F_{(\boldsymbol\gamma, \boldsymbol\xi)}$is said to be {\it $l$-th pseudo-conical} at the origin if there exist $h:(\mathbb{R},0)\to \mathbb{R}^3$ and a positive integer $l$ such that
\[
\bm{s}'(x) = h(x)x^l \quad h(0) \neq 0.
\]
We assume that there exists a striction curve of $F_{(\boldsymbol\gamma, \boldsymbol\xi)}$. We call $F_{(\boldsymbol\gamma, \boldsymbol\xi)}$ a {\it $(k,l)$-ruled surface} if it is $k$-th pseudo-cylindrical and $l$-th pseudo-conical.
%%%%%%%%%%%%%%%%%%%%%%%%%%%%%%%%%%%%%%%%%%%%%%%%%%%%%%%%%%%%%%%%%%%%%%%%%%
\begin{figure}[ht]
\centering
\includegraphics[width=0.8\textwidth]{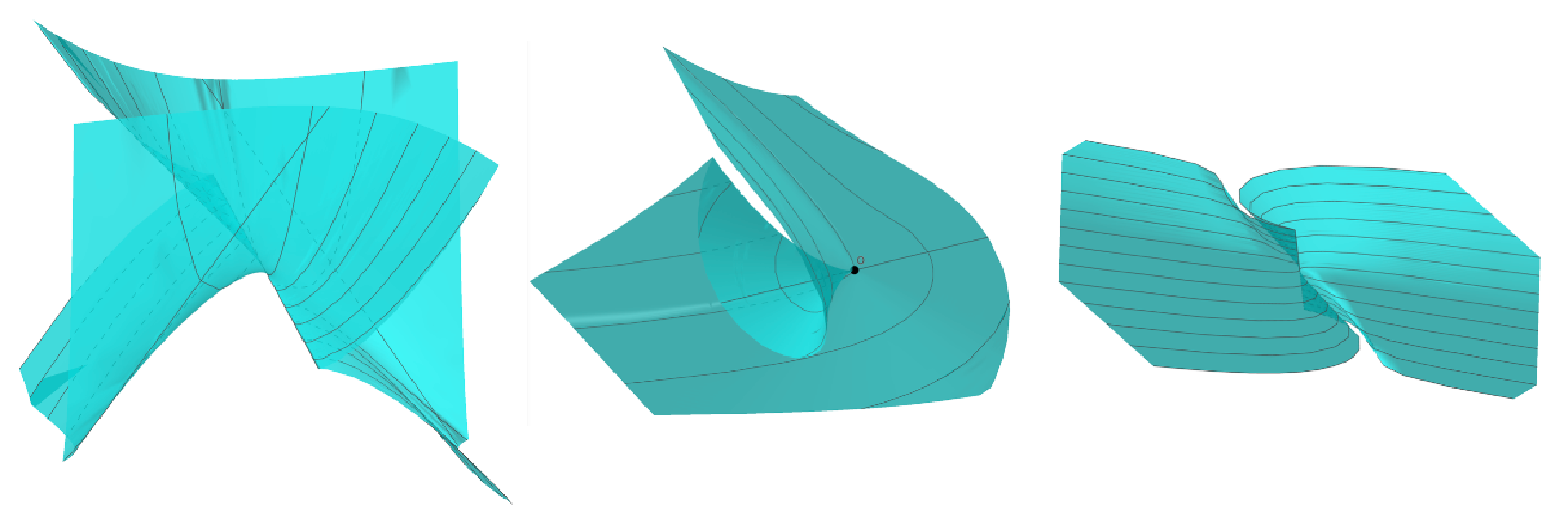}
\caption{A $(1,0)$-ruled surface (left), a $(0,1)$-ruled surface (center), and a $(1,1)$-ruled surface (right)}
\label{pseudo}
\end{figure}
%%%%%%%%%%%%%%%%%%%%%%%%%%%%%%%%%%%%%%%%%%%%%%%%%%%%%%%%%%%%%%%%%%%%%%%%%%

\subsection{{Developable surfaces along a curve passing through Whitney umbrellas}}%%%%%%%%%%%%%%%%%%%%%%%%%%%%%%%%%%%%%%
Let $W$ and $c_w$ be as defined in \eqref{normalWU} and \eqref{c_w} respectively. By \eqref{ebn}, \eqref{frenet}, and {Lemma \ref{k2k3}}, we can take a Darboux frame $\{ \bm{e}(x), \bm{b}(x), \bm{n}(x) \}$ and the function $\kappa_1, \kappa_2, \kappa_3 : (\mathbb{R}, 0) \to \mathbb{R}$ such that
\begin{gather}
\label{condition-1}
\gamma(x)=(W \circ c_w)(x) =
\mathcal{E}(x)x^{\alpha_0}, \ \mathcal{E}(0) \neq 0,\\
\label{condition-2}
\kappa_1(x) = \tilde{\kappa}_1(x)x^{\alpha_1},\
\kappa_2(x) = \tilde{\kappa}_2(x)x^{\alpha_2},\ 
\kappa_3(x) = \tilde{\kappa}_3(x)x^{\alpha_3},
\end{gather}
where $(\tilde\kappa_2(0),\tilde{\kappa}_3(0)) \neq (0,0)$, $\alpha_i \in \mathbb{Z}_{\geq 0}$ $(i = 0,1,2,3)$, $\mathcal{E} : (\mathbb{R}, 0) \to \mathbb{R}^3$, and $\tilde\kappa_i : (\mathbb{R} ,0) \to \mathbb{R}$ $(i=1,2,3)$. Then we may assume that
\begin{enumerate}
\item if $\alpha_3>\alpha_2$, then $\tilde\kappa_2(0)\ne0$,
\item if $\alpha_2>\alpha_3$, then $\tilde\kappa_3(0)\ne0$.
\end{enumerate}

In \cite{IO2015}, an osculating developable surface along a curve on a surface with a Darboux frame is defined. Following this construction, we define a map $OD_w : (\mathbb{R}^2, 0) \to (\mathbb{R}^3, 0)$ by
\[
OD_w (x,y) = \gamma(x) + yD_o(x),
\]
where
\[
D_o(x) =
\begin{cases}
	\displaystyle\frac{\tilde\kappa_3(x)\bm{e}(x)
	-\tilde\kappa_2(x)x^{\alpha_2-\alpha_3}\bm{b}(x)}
	{\sqrt{\tilde\kappa_2(x)^2x^{2(\alpha_2-\alpha_3)}
	+\tilde\kappa_3(x)^2}}
	\quad (\text{if}\ \alpha_2 >\alpha_3), \vspace{4mm} \\
	\displaystyle\frac{\tilde\kappa_3(x)x^{\alpha_3-\alpha_2}\bm{e}(x)
	-\tilde\kappa_2(x)\bm{b}(x)}
	{\sqrt{\tilde\kappa_2(x)^2
	+\tilde\kappa_3(x)^2x^{2(\alpha_3-\alpha_2)}}}
	\quad (\text{if}\ \alpha_3 \geq \alpha_2).
\end{cases}
\]
The surface $OD_w$ is a ruled surface and tangent to $W$ along $\gamma$. Define
\begin{align}\label{delta}
\delta =
\begin{cases}
	\tilde\kappa_1x^{\alpha_1}
	({\tilde{\kappa}_2}^2 x^{2(\alpha_2 -\alpha_3)} +{\tilde{\kappa}_3}^2)
	+{\tilde{\kappa}_2} x^{\alpha_2 -\alpha_3}{\tilde{\kappa}_3}'
	- ({\tilde{\kappa}_2} x^{\alpha_2 -\alpha_3})'{\tilde{\kappa}_3}
	\quad (\text{if}\ \alpha_2 >\alpha_3), \\
	\tilde\kappa_1x^{\alpha_1}
	({\tilde{\kappa}_2}^2 +{\tilde{\kappa}_3}^2x^{2(\alpha_3-\alpha_2)})
	+\tilde\kappa_2(\tilde\kappa_3x^{\alpha_3-\alpha_2})'
	- {\tilde\kappa_2}'\tilde\kappa_3x^{\alpha_3-\alpha_2}
	\quad (\text{if}\ \alpha_3 \geq \alpha_2),
\end{cases}
\end{align}
we have
\begin{equation}\label{D_o'}
{D_o}' = 
\begin{cases}
	\displaystyle\frac{\delta}
	{\{ {\tilde\kappa_2}^2 x^{2(\alpha_2-\alpha_3)}+{\tilde\kappa_3}^2\}
	^{3/2}} (\tilde{\kappa}_2 x^{\alpha_2 -\alpha_3} \bm{e}
	+\tilde{\kappa} _3  \bm{b})
	\quad (\text{if}\ \alpha_2 >\alpha_3), \\
	\displaystyle\frac{\delta}
	{\{ {\tilde\kappa_2}^2+{\tilde\kappa_3}^2x^{2(\alpha_3-\alpha_2)}\}
	^{3/2}} (\tilde{\kappa}_2\bm{e}
	+\tilde{\kappa} _3x^{\alpha_3-\alpha_2}\bm{b})
	\quad (\text{if}\ \alpha_3 \geq \alpha_2).
\end{cases}
\end{equation}
Here and in what follows, we omit $(x)$ for functions of variable $x$. By \eqref{D_o'}, we have $\det (\gamma', D_o, {D_o}') \equiv0$. From this, we see that $OD_w$ is developable. We call $OD_w$ an {\it osculating developable surface} of $W$ along $\gamma$. By \eqref{D_o'}, $OD_w$ is $k$-th pseudo-cylindrical if and only if there exists $\tilde{\delta} : (\mathbb{R}, 0) \to \mathbb{R}$ such that
\[
\delta(x) = \tilde\delta(x) x^k \quad \tilde\delta(0) \neq 0.
\]

\begin{prop}
If $OD_w$ is $k$-th pseudo-cylindrical, and satisfies that
\[
\begin{cases}
\alpha_0+\alpha_2-\alpha_3-1 \geq k \quad &(\rm{if}\ \alpha_2>\alpha_3),\\
\alpha_0-1 \geq k \quad &(\rm{if}\ \alpha_3 \geq \alpha_2),
\end{cases}
\]
then there exists a striction curve of $OD_w$.
\end{prop}
\begin{proof}
Let $OD_w$ be $k$-th pseudo-cylindrical, and $\bm{s}_w : (\mathbb{R},0) \to \mathbb{R}^3$ be a map
\[
\bm{s}_w(x)=OD_w(x,S(x))=\gamma(x) +S(x)D_o(x).
\]
By the equation \eqref{striction_eq}, if $\bm{s}_w$ is a striction curve, then it holds that
\[
\langle \gamma' +S'D_o +S {D_o}' , {D_o}' \rangle =0.
\]
By a straightforward calculation, it holds that 
\[
S(x) = 
\begin{cases}
	-|\mathcal{E}|\tilde\kappa_2
	\sqrt{{\tilde\kappa_2}^2 x^{2(\alpha_2-\alpha_3)}
	+{\tilde\kappa_3}^2}
	x^{\alpha_0+\alpha_2-\alpha_3-k-1}/{\tilde\delta}
	\quad &(\text{if}\ \alpha_2 >\alpha_3),  \\
	-|\mathcal{E}|\tilde\kappa_2
	\sqrt{{\tilde\kappa_2}^2
	+{\tilde\kappa_3}^2 x^{2(\alpha_3-\alpha_2)}}
	x^{\alpha_0-k-1}/{\tilde\delta}
	\quad &(\text{if}\ \alpha_3 \geq \alpha_2).
\end{cases}
\]
If the {degree} of $S(x)$ is zero or positive, we can define $\bm{s}_w$ at $x=0$. 
\end{proof}
Let $OD_w$ be $k$-th pseudo-cylindrical and {let} $\bm{s}_w$ be the striction curve of $OD_w$. If $OD_w$ satisfies
\[
\begin{cases}
\alpha_0+\alpha_2-\alpha_3-1 > k \quad &(\text{if}\ \alpha_2>\alpha_3),\\
\alpha_0-1 > k \quad &(\text{if}\ \alpha_3 \geq \alpha_2),
\end{cases}
\]
then $\bm{s}_w(0)=(0,0,0)^T$. That is, $\bm{s}_w$ passes through the Whiney umbrella singularity. Hence we consider this case. Then we obtain
\begin{equation}\label{s'}
{\bm{s}_w}'= \sigma D_o,
\end{equation}
where
\begin{equation}\label{sigma}
\sigma = 
\begin{cases}
\displaystyle
\frac{|\mathcal{E}| \tilde{\kappa}_3 x^{\alpha_0-1}}{\sqrt{{\tilde{\kappa}_2}^2 x^{2(\alpha_2 -\alpha_3)} +{\tilde{\kappa}_3}^2}}+S'
\quad (\text{if}\ \alpha_2>\alpha_3), \vspace{4mm}\\
\displaystyle
\frac{|\mathcal{E}| \tilde{\kappa}_3 x^{\alpha_0+\alpha_3-\alpha_2-1}}
{\sqrt{{\tilde{\kappa}_2}^2
+{\tilde{\kappa}_3}^2x^{2(\alpha_3 -\alpha_2)}}}+S'
\quad (\text{if}\ \alpha_3 \geq \alpha_2).
\end{cases}
\end{equation}
By \eqref{s'}, $OD_w$ is $l$-th pseudo-conical if and only if there exists $\tilde{\sigma} : (\mathbb{R}, 0) \to \mathbb{R}$ such that
\[
\sigma(x) = \tilde\sigma(x) x^l \quad \tilde\sigma(0) \neq 0.
\]
That is, if the degree of $\delta$ is $k$ and the degree of $\sigma$ is $l$, then $OD_w$ is a $(k,l)$-developable surface.

\subsection{Degrees and top-terms of $\delta$ and $\sigma$}%%%%%%%%%%%%%%%%%%%%%%%%%%%%%%%%%%%%%%%%%%%%
The functions $\delta$ and $\sigma$ are invariants of a developable surface. It is known that these invariants have geometric meanings in classifying singular points of the developable surface \cite{ISTa2017, ISTe2022}. Therefore, the degrees and top-terms of $\delta$ and $\sigma$ may have geometric meanings of the Whitney umbrella. We calculate the degrees and the top-terms. Moreover, we provide conditions under which the top-terms vanish. We assume the conditions \eqref{condition-1} and \eqref{condition-2}. 

We set the following three conditions.
\begin{enumerate}
\item [(i)] $\alpha_1<|\alpha_3-\alpha_2|-1$,
\item [(ii)] $\alpha_1=|\alpha_3-\alpha_2|-1$,
\item [(iii)] $\alpha_1 > |\alpha_3-\alpha_2|-1$.
\end{enumerate}

If $\alpha_3 > \alpha_2$ holds, then, from \eqref{delta}, it follows that
\[
\delta = 
\begin{cases}
\tilde\delta x^{\alpha_1}, \quad 
\tilde\delta=\tilde\kappa_1\tilde\kappa_2^2+O(x)
&\text{if}\ (\mathrm{i})\ \text{holds},\\
\tilde\delta x^{\alpha_1}, \quad
\tilde\delta=\tilde\kappa_2
(\tilde\kappa_1\tilde\kappa_2-(\alpha_3 -\alpha_2)\tilde\kappa_3)+O(x)
&\text{if}\ (\mathrm{ii})\ \text{holds}, \\
\tilde\delta x^{\alpha_3 -\alpha_2}, \quad 
\tilde\delta=-(\alpha_3 -\alpha_2)\tilde\kappa_2\tilde\kappa_3+O(x)
&\text{if}\ (\mathrm{iii})\ \text{holds}.
\end{cases}
\]
We assume $\tilde\delta(0) \neq 0$. Then, from \eqref{sigma}, we have
\[
\sigma =
\begin{cases}
\tilde\sigma x^{\alpha_0-\alpha_1-2}, \quad
\tilde\sigma=(\alpha_0-\alpha_1-1)
|\mathcal{E}||\tilde\kappa_2|\tilde\kappa_2/\tilde\delta +O(x)
&\text{if}\ (\mathrm{i})\ \text{holds},\\
\tilde\sigma x^{\alpha_0-\alpha_1-2}, \quad
\tilde\sigma =(\alpha_0-\alpha_1-1)
|\mathcal{E}||\tilde\kappa_2|\tilde\kappa_2/\tilde\delta +O(x)
&\text{if}\ (\mathrm{ii})\ \text{holds},\\
\tilde\sigma x^{\alpha_0-\alpha_3+\alpha_2-2}, \quad
\tilde\sigma=(\alpha_0-\alpha_3+\alpha_2-1)
|\mathcal{E}||\tilde\kappa_2|\tilde\kappa_2/\tilde\delta +O(x)
&\text{if}\ (\mathrm{iii})\ \text{holds}.
\end{cases}
\]

If $\alpha_2 = \alpha_3$ holds, then it follows that
\[
\delta =
\tilde\kappa_1x^{\alpha_1}
({\tilde{\kappa}_2}^2 +{\tilde{\kappa}_3}^2)
+{\tilde{\kappa}_2}{\tilde{\kappa}_3}'
- {\tilde{\kappa}_2}'{\tilde{\kappa}_3}.
\]
There exists an integer $k$ such that
\[
0 \le k \le \alpha_1, \quad
\delta = \tilde\delta x^k, \quad
\tilde\delta(0) \neq 0.
\]
We assume that $\tilde\delta(0) \neq 0$. We have
\[
\sigma =
\tilde\sigma x^{\alpha_0-k-2}, \quad
\tilde\sigma = (\alpha_0-k-1)
|\mathcal{E}||\tilde\kappa_2|\tilde\kappa_2/\tilde\delta +O(x).
\]
By the case of $\alpha_3>\alpha_2$ and $\alpha_3=\alpha_2$, the top-term of $\sigma$ does not vanish.

If $\alpha_2 > \alpha_3$ holds, then it follows that
\[
\delta = 
\begin{cases}
\tilde\delta x^{\alpha_1}, \quad 
\tilde\delta=\tilde\kappa_1\tilde\kappa_3^2+O(x)
&\text{if}\ (\mathrm{i})\ \text{holds},\\
\tilde\delta x^{\alpha_1}, \quad
\tilde\delta=\tilde\kappa_3
(\tilde\kappa_1\tilde\kappa_3-(\alpha_2 -\alpha_3)\tilde\kappa_2)+O(x)
&\text{if}\ (\mathrm{ii})\ \text{holds}, \\
\tilde\delta x^{\alpha_2 -\alpha_3}, \quad 
\tilde\delta=-(\alpha_2 -\alpha_3)\tilde\kappa_2\tilde\kappa_3+O(x)
&\text{if}\ (\mathrm{iii})\ \text{holds}.
\end{cases}
\]
We assume $\tilde\delta(0) \neq 0$. We have
\[
\sigma =
\begin{cases}
\tilde\sigma x^{\alpha_0-1}, \quad
\tilde\sigma(0)=\tilde\kappa_3|\mathcal{E}|/|\tilde\kappa_3|+O(x)
&\text{if}\ (\mathrm{i})\ \text{holds},\\
\tilde\sigma x^{\alpha_0-1}, \quad
\tilde\sigma ={|\mathcal{E}||\tilde\kappa_3|}
\{ \tilde\kappa_1\tilde\kappa_3 
-(\alpha_0+\alpha_2-\alpha_3)\tilde\kappa_2\}/{\tilde\delta}+O(x)
&\text{if}\ (\mathrm{ii})\ \text{holds},\\
\tilde\sigma x^{\alpha_0-2}, \quad
\tilde\sigma=(\alpha_0-1)
|\mathcal{E}||\tilde\kappa_3|\tilde\kappa_2/\tilde\delta +O(x)
&\text{if}\ (\mathrm{iii})\ \text{holds}.
\end{cases}
\]
We assume the condition $(\mathrm{ii})$. Then the top-terms of $\delta$ and $\sigma$ are given by the following coefficients.
\[
E=\tilde\kappa_1\tilde\kappa_3-(\alpha_2 -\alpha_3)\tilde\kappa_2,
\]
\[
F=\tilde\kappa_1\tilde\kappa_3-(\alpha_0+\alpha_2-\alpha_3)\tilde\kappa_2.
\]
If $E=0$ (respectively, $F=0$), then the top-term of $\delta$ (respectively, $\sigma$) vanishes.

\begin{ex}
We give an example which may have the relation $E =0$ and $F =0$. We consider the following condition:
\[
c_w(x) =
(c(x)x^{2m}, x^m),\ 
c(x) = c_0 +c_mx^m +O(x)^{m+1},
\]
\[
A = 6a_{11} c_0 ^2 +a_{03}c_0 -3a_{02} c_m,\ 
B = 3c_0 +\frac{b_3}{2} =0,
\]
\[
C = -4c_0 ^2-a_{02} ^2 \neq 0,\ 
D=(a_{11}b_3-a_{03})c_0 -5a_{02}c_m -\frac{b_4 a_{02}}{3}.
\]
Then it holds
\begin{align}
&\kappa_1 = \tilde\kappa_1x^{m-1}, \quad
\tilde\kappa_1= \frac{m^3a_{02}}{|\mathcal{E}|^2|\mathcal{N}|}A +O(x),
\nonumber \\
&\kappa_2 = \tilde\kappa_2 x^{2m-1}, \quad
\tilde\kappa_2= \frac{m^3}{|\mathcal{E}||\mathcal{N}|}D +O(x),
\nonumber \\
&\kappa_3 = \tilde\kappa_3 x^{m-1}, \quad
\tilde\kappa_3= \frac{m^3a_{02}}{|\mathcal{E}||\mathcal{N}|^2}C +O(x).
\nonumber
\end{align}
This satisfies the condition $(\mathrm{ii})$ (that is, $\alpha_1 = \alpha_2 -\alpha_3 -1$).
Calculating $E$ and $F$, for $x=0$, we obtain
\[
E=\frac{2m^3}{|\mathcal{E}||\mathcal{N}|}
( 6a_{11} c_0 ^2 +a_{03}c_0 +a_{02} c_m +\frac{b_4 a_{02}}{6}),
%=\frac{4m^3}{|\mathcal{E}||\mathcal{N}|}(A+4a_{02}c_m),
\]
\[
F=\frac{3m^3}{|\mathcal{E}||\mathcal{N}|}
( 24a_{11} c_0 ^2 +4a_{03}c_0 +12a_{02} c_m +b_4 a_{02}).
%=\frac{4m^3}{|\mathcal{E}||\mathcal{N}|}(A+6a_{02}c_m).
\]
%We obtain that $E$ (respectively, $F$) satisfies zero if and only if it holds $A=-4a_{02}c_m$ (respectively, $A=-6a_{02}c_m$). Since The formula $A$ is the top-term of $\kappa_1$, the top-terms of $\delta$ and $\sigma$ depend on the top-term of $\kappa_1$.
\end{ex}

\end{document}